\documentclass[12pt]{amsart}
\usepackage[sans]{dsfont}
\usepackage{amsfonts,amssymb,amsbsy,amsmath,amsthm,enumerate}

\numberwithin{equation}{section}

\newtheorem{theorem}{Theorem}[section]
\newtheorem{proposition}[theorem]{Proposition}
\newtheorem{lemma}[theorem]{Lemma}
\newtheorem{defi}[theorem]{Definition}
\newtheorem{corollary}[theorem]{Corollary}

\theoremstyle{definition}
\newtheorem{remark}[theorem]{Remark}

\usepackage{amsfonts}
\usepackage{amsmath}
\usepackage{bigints}
\usepackage{bm}
\usepackage{dsfont}
\usepackage{color}

\definecolor{gr}{rgb}   {0.,   0.69,   0.23 }
\newcommand{\Bk}{\color{black}}

\def\beq{\begin{equation}}
\def\eeq{\end{equation}}
\newcommand{\bea}{\begin{eqnarray}}
\newcommand{\eea}{\end{eqnarray}}
\newcommand{\beas}{\begin{eqnarray*}}
\newcommand{\eeas}{\end{eqnarray*}}

\newcommand{\bel}{\begin{equation} \label}
\newcommand{\ee}{\end{equation}}

\newcommand{\bethl}{\begin{theorem} \label}

\newcommand{\beprl}{\begin{proposition} \label}
\newcommand{\epr}{\end{proposition}}

\newcommand{\belel}{\begin{lemma} \label}
\newcommand{\ele}{\end{lemma}}

\newcommand{\becol}{\begin{corollary} \label}
\newcommand{\eco}{\end{corollary}}

\newcommand{\bepf}{\begin{proof}}
\newcommand{\epf}{\end{proof}}

\newcommand{\one}{\mathds{1}}

\newcommand{\capa}{{{\rm Cap}}}

\newcommand{\rd}{{\mathbb R}^{2}}

\newcommand{\re}{{\mathbb R}}

\newcommand{\R}{{\mathbb R}}

\newcommand{\C}{{\mathbb C}}

\newcommand{\N}{{\mathbb N}}
\newcommand{\Z}{{\mathbb Z}}

\newcommand{\gC}{{\mathfrak{C}}}

\newcommand{\gM}{{\mathfrak{M}}}

\newcommand{\cC}{{\mathcal C}}

\newcommand{\cE}{{\mathcal E}}

\newcommand{\cK}{{\mathcal K}}

\newcommand{\cV}{{\mathcal V}}

\begin{document}


\title[2D Magnetic Dirac operator]{Spectrum of the perturbed Landau-Dirac operator 
}
\author{Vincent Bruneau and  Pablo Miranda}

\begin{abstract}\setlength{\parindent}{0mm}
In this article, we consider  the Dirac operator  with constant magnetic field in $\R^2$. Its spectrum consists of eigenvalues of infinite multiplicities, known as  the Landau-Dirac levels.
  Under  compactly supported perturbations {of the electric and magnetic potentials}, we study the distribution of the discrete eigenvalues near each Landau-Dirac level.  Similarly to  the Landau (Schr\"odinger) operator, we demonstrate  that a  three-term asymptotic formula holds for  the eigenvalue counting function.  One of the main novelties of this work is the treatment of some perturbations  of variable sign. In this context  we explore some remarkable phenomena related to the finiteness or infiniteness of the discrete eigenvalues, which depend on the interplay of the different terms in the matrix perturbation.

\end{abstract}

\maketitle

{\centerline{\it Dedicated to the memory of Georgi Raikov.}

$$ $$



\section{The $2D$ Magnetic Landau-Dirac operators}

We consider in $L^2(\R^2)^2$ the Dirac operator with the homogeneous magnetic field $B=(0,0,b)$ pointing in the direction perpendicular to the plane, where $b>0$ is the amplitude.   More specifically, let $\sigma:=(\sigma_1,  \sigma_2)$, $ \sigma_3$ be the Pauli matrices,
 $A:=(A_1,A_2)=b(-\frac{ x_2}2, \frac{ x_1}2)$  a magnetic potential associated with the constant magnetic field  (curl$A=b$) and $m\geq 0$  the mass.
Then,  the {\it Landau-Dirac} operator $D_0$ is \Bk  the closure on  $C_0^\infty(\R^2)^2$  of 
$$ \sigma \cdot (-i \nabla -A) + m \sigma_3= \left( \begin{array}{cc}  m & a^*  \\
a & - m
\end{array} \right), $$
 where 
\bel{defa}
a := (-i \partial_{x_1} -A_1) + i (-i \partial_{x_2} -A_2) , \qquad a^* := (-i \partial_{x_1} -A_1) - i (-i \partial_{x_2} -A_2). 
\ee

 It is well known (see for instance \cite[Theorem 7.2]{Th92}) that the spectrum of $D_0$ is made up of eigenvalues of infinite multiplicities, the so-called {\it Landau-Dirac} Levels  
$$
\mu_{q}:=\begin{cases}
\sqrt{2bq+m^2},& q \in \{0,1,2,...\}\\[0.5em]
-\sqrt{2b|q|+m^2},&q \in \{-1,-2,... \}.
\end{cases}
$$
Thus the spectrum of $D_0$ is essential, given by
$\sigma(D_0) =  \sigma_{\text{ess}} (D_0)  = \{ \mu_{q} , \; q \in\Z \}$ and we have 
\bel{R0z}
(D_0-z)^{-1} =  \sum_{q \in \Z} (\mu_{q}-z)^{-1} {\mathcal P}_{q}   ,
\ee
where, for $q\in\Z$,
${\mathcal P}_{q}$ is  the  infinite dimensional orthogonal projection onto Ker$(D_0-\mu_{q})$.  These properties follow from analogous ones for the Landau-Schr\"odinger operator $H_L:= (-i \nabla -A)^2$ whose spectrum is given by the {\it Landau levels} $\Lambda_n = b(2n+1)$, $n \in \Z_+:=\{0,1,2,...\}$, by using  supersymmetry properties (see e.g. \cite[Theorem 5.13]{Th92}).

Let  $\mathcal V$ be a { Hermitian} matrix-valued perturbation, that is for $V_1,V_2$, and $W$  in $L^\infty(\R^2)$, 
${\cV}$ is the  operator of multiplication defined by 
\bel{defV}
\cV:=\begin{pmatrix}
	V_1&W^*\\
	W&V_2
\end{pmatrix} = \frac{V_1 + V_2}2 I_2 + \sigma_3 \frac{V_1 - V_2}2 - \sigma_1 \widetilde{A}_1 - \sigma_2 \widetilde{A}_2,   
\ee
{where $W =  -\widetilde{A}_1 -  i  \widetilde{A}_2$}.
Such matrix allows to consider scalar electric perturbations (with $V_1=V_2$, $W=0$), Lorentz potentials (with $V_1=-V_2$, $W=0$) and  perturbations by magnetic fields { $\widetilde{b}= \partial_{x_1} \widetilde{A}_2 - \partial_{x_2} \widetilde{A}_1$} associated with the magnetic potential  {$(\widetilde{A}_1, \widetilde{A}_2)$
} (with $V_1=V_2=0$). 

Now, for $V_1,V_2$, and $W$  tending to $0$ at infinity, ${\cV}$ is a relatively compact perturbation of  $D_0$. By   Weyl's theorem, the essential spectrum of the self-adjoint operator 
$D_{\cV}:= D_0 + \cV$ satisfies:
$$ 
 \sigma_{\text{ess}}(D_{\cV})=\sigma_{\text{ess}}(D_{0})= \{ \mu_{q} , \; q \in \Z \}.$$

Thus, under the  perturbation  $\cV$, discrete  eigenvalues may appear, with the only possible limit points at the  Landau-Dirac levels. 

In this paper we will focus on compactly supported perturbations ${\cV}$  and for each fixed $q\in \Z$, $0<\lambda$, we will consider the counting functions
 \bel{defNpm}
 {\mathcal N}_{+}^q(\lambda):={\rm Tr} E_{D_0+{\cV}}(\mu_q+\lambda,\alpha),\quad {\mathcal N}_{-}^q(\lambda):= {\rm Tr}E_{D_0+{\cV}}(\alpha,\mu_q-\lambda), 
 \ee
  where $E_T(\omega)$ is the spectral projection of the self-adjoint operator $T$  associated with the Borel set $\omega$, and $\alpha$ is any fixed number in $(\mu_{q},\mu_{q+1})$ and  $(\mu_{q-1},\mu_{q})$, respectively. These  functions  count the number of eigenvalues of $D_0+{\cV}$, including  multiplicities, on the intervals $(\mu_q+\lambda,\alpha)$ and $(\alpha,\mu_q-\lambda)$, respectively.

{
The asymptotic behavior of these functions was studied before in \cite{MeRo07} in the more restrictive case of ${\cV}$ with definite sign and of form ${\cV}=I_2 v$, with $I_2$ the identity $2\times 2$ matrix. In \cite[Theorem 1.3]{MeRo07} it is proved that}
 the eigenvalue counting functions satisfy
$${\mathcal N}_{\pm}^q(\lambda)=\frac{|\ln{\lambda}|}{ \ln{|\ln{\lambda}|}}(1+o(1)),\quad \lambda\downarrow 0.
$$ 
Analogous results for the Landau-Schr\"odinger operator $H_L +v$ in $\R^2$,  were obtained  in   \cite[Theorem 2.2]{RaWa02} and \cite[Theorem 1.2]{MeRo07}. 
In these works it is shown that for $H_L +v$ the asymptotic distribution of the eigenvalues near the Landau levels is governed by the Toeplitz operators $p_n  vp_n$  where $p_n $ is the orthogonal projection onto ${\rm Ker}(H_L-\Lambda_n)$, and by  ${\mathcal P}_{q} V  {\mathcal P}_{q}$ for $D_0+{\cV}$.

In \cite{FiPu06}, a different approach allows to describe the rate of accumulation of the  eigenvalues 
of $p_n v p_n$ 
in terms of the logarithmic capacity of $K$, the support of $v$. Then, this geometrical quantity appears in the third term of the asymptotic expansion of the eigenvalue counting function  of $H_L +v$. Similarly,  this result also holds for the Dirichlet, Neumann or Robin realization of $(-i \nabla -A)^2$ on $\R^2 \setminus K$, $K$ being a non-empty obstacle (i.e.  $K \subset \R^2$ is compact with a non-empty interior). The Dirichlet  case is like $H_L + \chi_K$ (see \cite{PuRo07}) while the obstacle problem for Neumann or Robin boundary conditions behaves like $H_L - \chi_K$ (see \cite{Pe09, GoKaPe16}). 
Here and throughout the article, $\chi_K $ denotes the characteristic function of the set $K$.
For the perturbation of $H_L$ by a delta interaction  supported on a curve $\Gamma  \subset \R^2$,  the logarithmic capacity of $\Gamma$ also appears in the asymptotic distribution of the eigenvalues near the Landau levels (see \cite{BeExHoLo20}).  In all mentioned works a fixed sign is assumed for the perturbation.

A natural question concerning the perturbation of the Landau-Dirac operator is the existence of analogous results on the
distribution of eigenvalues near $\mu_{q}$, $q \in \Z$. 
Moreover, we would like  also to be able to consider potentials of no-definite sign (in particular to take into account the Lorentz potential, certain magnetic perturbations as well as for the purpose of studying obstacle problems). {Some results for non-definite sign perturbations of the Landau-Schr\"odinger operators  are considered in  \cite{PuRo11} (for electric potentials of compact support) and in \cite{RoTa08}  {  (with smooth compactly supported magnetic fields under a sign condition on an effective potential)}. However, it is important to notice  that in order to use ideas from  these works in our context, the first difficulty  to overcome  is the non-semiboundedness of the Dirac operator.}

In this article, we obtain a three-term asymptotic expansion of  ${\mathcal N}_{\pm}^q(\lambda)$     for compactly supported perturbations of $D_0$ that  are not of definite sign (see \eqref{eqthm1}).  To achieve this, we first {obtain some general}  results concerning the index of a pair of spectral projections for non-semibounded operators (see the Appendix). Using the concept of spectral flow, this index helps to generalize a version of the  Birman-Schwinger principle to {non-sign-definite} perturbations 
 (see e.g. \cite{Pu09}). Then, it reduces the study of the counting functions \eqref{defNpm} to the analysis of the distribution of the eigenvalues of some Toeplitz operators ${\mathcal P}_{q} {\cV}^\pm_\epsilon  {\mathcal P}_{q} $ with $ {\cV}^\pm_\epsilon = {\cV}  + O(\epsilon)  $ for  $\epsilon>0$ small enough (see Proposition \ref{Diag}).  { The spectral study of these operators is based on complex analysis tools developed in \cite{FiPu06} and exploits \cite{PuRo11}}.

The main asymptotic expansions obtained are analogous to those  given  for fixed sign  perturbations  of the Landau-Schr\"odinger operator. However, a surprising result is that {when the support of $V_1$, $V_2$ and $W$ coincide, only $V_1$ determines the side of the accumulation for the eigenvalues} 
(see in particular Theorem \ref{thm1} as well as Remarks \ref{rmqV1} and \ref{rmqV1+}). {We show that this predominance can be attenuated and compensated by $V_2$ when the support of $V_2$ is not { contained} in the support of $V_1$ 
(see Theorem \ref{theo1_b}). 
For smooth potentials, near each $\mu_q$, $q \in \Z$, we give an effective Toeplitz operator $p_0 v_q({\cV}) p_0$ with $v_q({\cV})$ a function given by a linear combinations of $V_1$, $V_2$, $\widetilde{b} = \partial_{x_1} \widetilde{A}_2 -  \partial_{x_2} \widetilde{A}_1$ and their derivatives until the order $2 |q|$ (see Proposition \ref{vqeff}). 
{ In contrast to \cite{RoTa08}, we have assumptions on the support of the magnetic potential, which is not very satisfactory because the physically relevant quantity is the magnetic field. However, our effective operator is expressed in terms of the electric potential and the magnetic field, and we can hope that the approach of \cite{RoTa08} allows to reduce the assumptions in terms of the magnetic field (independently of the choice of gauge).} Unlike our previous results, with this type of effective potential it is difficult to {express in a simple manner} the influence of the different { components} of ${\cV}$ {over the asymptotics of  ${\mathcal N}_{\pm}^q$}. However, this allows us to establish another distinctive feature: it is possible to have very different { patterns of} behavior between two successive Landau-Dirac levels (see {Corollary} \ref{prop_inf}).

\section{Results}\label{s2}

 As we already mentioned, the asymptotic distribution of discrete eigenvalues of $D_0+{\cV}$ is related to a particular geometric feature of the support of ${\cV}$. In order to capture this properties we first need the following  definitions taken from potential theory.
 
 \subsection{Logarithmic capacity} Let $\cE \subset \rd$  be a Borel set, and $\gM(\cE)$ denote the set of compactly supported probability measures on $\cE$. Then  the  {\em logarithmic capacity} of $\cE$  is defined as
$\capa(\cE) : = e^{-{\mathcal I}(\cE)}$ where
$$
{\mathcal I}(\cE) : = \inf_{\mu \in \gM(\cE)}  \int_{\cE \times \cE} \ln{|x-y|^{-1}} d\mu(x) d\mu(y),
$$
{(if ${\mathcal I}(\cE)= + \infty$, we set $\capa(\cE) =0$).}
A systematic exposition of the theory of the logarithmic capacity can be found, for example, in \cite[Chapter 5]{Ran95} and \cite[Chapter II, Section 4]{La72}. 
If we identify $\cE$ with a subset of $\C$ (with $z=x_1+ i x_2$ for $(x_1,x_2) \in \cE$), for bounded sets, the  logarithmic capacity coincides with the  {\em transfinite diameter}
$${\rm Cap} (\cE) = \lim_{n \to \infty} \delta_n(\cE)$$ where
$$\delta_n(\cE) := \max_{z_1, \cdots, z_n \in \cE} \Big( \prod_{1 \leq i <j\leq n} |z_i - z_j| \Big)^{\frac2{n(n-1)}}.
$$
It is also given by the $n$'th root asymptotic of the $n$'th {\em Chebyshev polynomial} for the set $\cE$:
$${\rm Cap} (\cE) = \lim_{n \to \infty} \Big( \sup_{z \in \cE} | f_n(z)|\Big)^{\frac1{n}}$$
where $f_n$ is {the} monic polynomial {that minimizes}  $\sup_{z \in \cE} | f(z)|$ among all the  monic polynomials of degree $n$
(see for instance \cite{StTo92, Sa10}).

\subsection{Weights for open sets}  For an open set  $\Omega$  in $\R^2$ we  will say that a { measurable} function $w:\Omega\to [0,\infty)$ is a {\em weight } in $\Omega$ if the weighted Sobolev space { $$H^1(\Omega;w):= \{ f \in L^2(\Omega;w) ; \; \nabla f \in L^2(\Omega;w) \}$$} is a Hilbert space and the embedding 
\bel{embe}
H^1(\Omega;w)\hookrightarrow L^2(\Omega;w)
\ee
is compact {(here $L^2(\Omega;w)$ is the weighted $L^2$-space with the positive measure $wdx$}). There exists a vast literature concerning weighted Sobolev spaces and conditions for the compactness of the embedding
\eqref{embe} (see \cite{GuOp91,An03,GoUk09} and references therein). Of course any constant positive function
is a weight over $\Omega$ if, for instance, $\Omega$ has a Lipschitz boundary. Another  simple example is the following: suppose that $\Omega$ is a bounded domain with  smooth  boundary $\partial \Omega$.  Then a positive function of the form 
$$
w(x)=|{\rm dist}(x,\partial \Omega)|^\alpha
$$
{near $\partial \Omega$}, is a weight for $\Omega$ if $\alpha\geq1$ (see for instance \cite{An03}). It is also possible to consider the distance to a part of the boundary (as in \cite[Section 2.1]{Ca08}).  For the compactness of \eqref{embe}, much less restrictive  conditions on the regulatity of $\Omega$ and $w$  appear in the aforementioned literature.

{As we will see later, the concept of weight allows us to extend,  to some extent, the framework of validity of the results in \cite{FiPu06}, upon which  part of our results are based. Specifically, this enables us to consider potentials that are not necessarily larger than a positive constant in a given region, but are instead  greater  than a weight $\omega$ in the same region.}

\subsection{Main results}  For a function {$v$  with compact essential support\footnote{ $ \operatorname {ess\,supp} (v):={\mathbb R}^2\setminus \bigcup \left\{{\mathcal O} \subseteq {\mathbb R}^2:{\mathcal O} {\text{ is open and }}v=0,\, {\text{almost everywhere in }}{\mathcal O}  \right\}$.\Bk}, denoted by $ {\rm ess\,supp} \,  v$, 
define
\bel{ms50+}
\gC_+(v) : = 1+  \ln{\left(\frac{b}2\,{\rm Cap}({\rm ess\,supp} \, v)^2\right)}.
\ee
}
For $v$  non--negative, set
$
K_-(v):=\{ z\in\C: \limsup_{r\downarrow 0} \frac{\log \int_{|z-\zeta|\leq r}v(\zeta)d m(\zeta)}{\log r}<\infty\}
$ where $m(\zeta)$ denotes the Lebesgue measure on $\R^2$, and 
define 
\bel{ms50-}
\gC_-(v) : = 1+  \ln{\left(\frac{b}2\,{\rm Cap}(K_-(v))^2\right)}.
\ee
Obviously since $K_-(v) \subset {\rm ess\,supp} \, v$, then $\gC_-(v)\leq \gC_+(v)$. Notice that the equality holds { for example,} when $v$ satisfies $C \chi_K \geq v \geq c \chi_K$ for some constants $C\geq c >0$ and a compact set $K \subset \C$ with non-empty interior.

{First, we show that when  $V_1>0$ on a domain $\Omega$  \Bk and its support contains that of $W$ and $V_{2,-}:= \max (0, -V_2)$, the distribution of the eigenvalues near each Landau-Dirac level is mainly governed by $V_1$ (in particular, the side where the eigenvalues accumulate is determined only by the sign of $V_1$).}

\begin{theorem}\label{thm1} 
 Let ${\cV} \in L^{\infty}(\R^2,\R^2 )$  be a {function given by \eqref{defV}  essentially supported in $\cK$, a compact set with Lipschitz boundary.} 
Suppose that there exist $\Omega$ an open bounded set in $\R^2$, $w$ a weight { in} $\Omega$ {and constants $C_2, C_3\geq 0$ such that:
  \bel{V1>}
 V_1 \geq \chi_\Omega w ; \quad  V_2 \geq - C_2 \chi_\Omega w ; \quad | W | \leq C_3 \chi_\Omega w .
  \ee}
 Then, for any  $q \in \Z$, the eigenvalue counting function satisfies
\begin{align}\label{eqthm1}\begin{split}
  &\frac{|\ln{\lambda}|}{\ln_2(\lambda)^2}\big(\gC_-(\chi_\Omega w)+o(1)\big)\\	
\leq&\,{\mathcal N}_+^q(\lambda)-\frac{|\ln{\lambda}|}{\ln_2(\lambda)} - \frac{|\ln{\lambda}|\ln_3(\lambda)}{\ln_2(\lambda)^2} \qquad  \lambda \searrow 0,\\
\leq&
  \frac{|\ln{\lambda}|}{\ln_2(\lambda)^2}\big(\gC_+(\chi_{\cK})+o(1)\big)
\end{split}
 \end{align}
where 
     $
     \ln_2(\lambda) : = \ln{|\ln{\lambda}|}$,  $\ln_3(\lambda) : = \ln{\ln_2(\lambda)}
    $.

On the other side 
\bel{eqthm1_b}
{\mathcal N}_-^q(\lambda)=O(1), \qquad  \lambda \searrow 0.
\ee
When $q=0$ the hypothesis on $V_2$ is useless and {in the upper bound, $\cK$ can be replaced by $\cK_1$, a compact set with Lipschitz boundary which contains ${\rm ess\,supp} \, V_1 \cup {\rm ess\,supp} \, W$.}

If (\ref{V1>}) is satisfied by $-{\cV}$, then (\ref{eqthm1}) is  satisfied by  
${\mathcal N}_-^q(\lambda)$ and (\ref{eqthm1_b}) by ${\mathcal N}_+^q(\lambda)$.
\end{theorem}

{
\begin{remark}\label{rmqV1}
Observe that the predominance of $V_1$ in the previous result contradicts   the intuitive expectation derived from  analyzing the symbol of the Dirac operator $D_{\cV}$. Indeed, the eigenvalues of $\sigma \cdot (\xi -A) + m \sigma_3 +{\cV}$, the matrix symbol of $D_{\cV}$, are given by
$$ \rho_\pm(x,\xi):= \frac{V_1 + V_2}2 \pm \sqrt{ | \xi - A+W|^2 + \left(\frac{V_1 - V_2}2 + m\right)^2}.$$
This formula suggests that the roles of $V_1$ and $V_2$ are { similar.} Unlike a semi-classical context of perturbative magnetic field (e.g. as in \cite{BrRo99}), Theorem \ref{thm1} shows that the presence of a constant magnetic field significantly modifies the way a perturbation influences the spectral behavior. A mathematical explanation of the origin of this phenomenon comes from Lemma \ref{lem1} (see Remark \ref{rmqV1+}).

\end{remark}
}

\begin{remark} It is worth  noting that for Lorentz potentials such as 
$$
 {\cV}_M := \begin{pmatrix}
	M \chi_\Omega &0\\
	0&-M  \chi_\Omega \end{pmatrix}, \quad M>0,$$
 {with $\Omega$ a bounded domain with Lipschitz boundary,} this result gives 
 the following asymptotic expansion 	as $ \lambda \searrow 0$
 \bel{conj}{\mathcal N}_+^q(\lambda) = \frac{|\ln{\lambda}|}{\ln_2(\lambda)} + \frac{|\ln{\lambda}|\ln_3(\lambda)}{\ln_2(\lambda)^2}  +   \frac{|\ln{\lambda}|}{\ln_2(\lambda)^2}\big(\gC(\Omega)+o(1)\big), \qquad {\mathcal N}_-^q(\lambda)=O(1),\ee
with $\gC(\Omega)  = \gC_-(M \chi_\Omega )  = \gC_+(M \chi_\Omega )  = 1+  \ln{\left(\frac{b}2\,{\rm Cap}(\Omega)^2\right)}$.

Having in mind that the ``infinite mass realization" of the Dirac operator in $\R^2 \setminus \overline{\Omega}$, $D^{MIT}_{\overline{\Omega}^c}$, is the limit (in the resolvent sense) of $D_0 + {\cV}_M$ as $M \nearrow + \infty$ (see for instance \cite{BCLS} in the absence of a magnetic field), the previous result suggests the following conjecture concerning the distribution of the eigenvalues of $D^{MIT}_{\overline{\Omega}^c}$ near each Landau-Dirac level $\mu_q$, $q \in \Z$.

\noindent
{\bf Conjecture. }  {\em The Dirac operator in $\R^2 \setminus \overline{\Omega}$ with \emph{Infinite Mass boundary condition} 
has an infinite  (resp. finite)  number of discrete eigenvalues above (resp. below) each Landau-Dirac level $\mu_q$ with the distribution given by \eqref{conj}.} { Or more precisely, 
$\left\{\lambda_{k,q}^+(\Omega)\right\}_{k \in \Z_+}$, the non-increasing sequence of eigenvalues above $\mu_q$  satisfies the asymptotic behavior:
\[	\lim_{k \to \infty} \left(k! (\lambda_{k,q}^+(\Omega) - \mu_q) \right)^{1/k} = \frac{b\, \capa(\Omega)^2}{2}.\]}
\end{remark}


For our second result both $V_1$ and $V_2$ could have negative contributions provided that the support of the positive contribution ``encircle" the support of the negative one. {The contribution of $W$ will be still weaker}. More specifically: 
\begin{defi}\label{def}
We say that a compact set  $K$ is \emph{encircled} by  an open set $\Omega$ if  there  exists a Jordan curve $\Gamma\subset\Omega $ such that $K$ is contained in the interior part of $\Gamma$.
\end{defi}
For a general class of open sets which encircle compacts set we refer to Proposition \ref{leborder} of the Appendix.

\begin{theorem}\label{theo1_b}
 Let ${\cV} \in L^{\infty}(\R^2)$  be a {function given by \eqref{defV}  essentially supported in $\cK$, a compact set with Lipschitz boundary.}
 Suppose there exist { $K \subset \cK \subset  \R^2$} a compact set, $\Omega_1, \Omega_2$ open bounded subsets of $\R^2\setminus K$, $w$ a continuous and positive weight for $\Omega_1$  and for $ \Omega:=\Omega_1 \cup \Omega_2$ such that:
 \bel{V=}
 V_1 \geq \chi_{\Omega_1} w - C \chi_{K}  ; \quad V_2 \geq \chi_{\Omega_2} w - C \chi_{K}  ; \quad | W | \leq C_3 \chi_{\Omega_1} w  + C \chi_{K},
 \ee 
for some constants $C, C_3 \geq 0$ and  $K$ is \emph{encircled} by $\Omega$.

  Then 
the asymptotics \eqref{eqthm1} and  \eqref{eqthm1_b} hold. {When $q=0$ the hypothesis on $V_2$ is still useless and $V_2$ does not appear in the upper bound.}

If (\ref{V=}) is satisfied by $-{\cV}$, then the above estimates are satisfied by interchanging ${\mathcal N}_+^q(\lambda)$ and ${\mathcal N}_-^q(\lambda)$.
 \end{theorem}
 
 Combining both above theorems, it is easy to see that Theorem \ref{theo1_b} still hold true if in \eqref{V=}, the assumption on $V_2$ is replaced by:
\bel{V=>}
 V_2 \geq \chi_{\Omega_2} w - C_2 \chi_{\Omega_1} w - C \chi_{K}
\ee
 for some constants $C_2 , C \geq 0$.
 
%


{The above results show a hierarchy between the { components} 
of ${\cV}$. If the supports of {$V_1$,$V_2$ and $W$ are  the same}, then the localization and the distribution of the eigenvalues are essentially given by the sign of $V_1$. In this case the distribution of the eigenvalues near each Landau-Dirac level is similar. Theorem \ref{theo1_b} shows that $V_2$ can play a predominant role  when, for instance, $\Omega_1 = \emptyset$ (except near $\mu_0$). In this case $D_{\cV}$ may have a finite number of eigenvalues near $\mu_0$ (e.g. for $V_1=0$) and an infinite number near each $\mu_q$, $q \in \Z^*$ {(the opposite property will be obtained in Corollary \ref{prop_inf}).} 

In the following section we 
will identify an effective operator allowing to explicitly analyze the competition between $V_1$, $V_2$ and $W$ {when these functions are sufficiently smooth} (in the spirit of \cite{RoTa08} for the Schr\"odinger and Pauli operators).

\subsection{ Effective operator for smooth perturbations }
%

In this section we assume that $V_1$, $V_2$ and $W$ are smooth compactly supported functions. Our main goal is to prove that the distribution of the eigenvalues of $D_{\cV}$ near $\mu_q$, $q \in \Z$, is governed by effective Toeplitz operator $p_0 v_q({\cV}) p_0$ with $v_q({\cV})$ { being} the smooth function given by the following linear combinations of $V_1$, $V_2$, $\widetilde{b} = \partial_{x_1} \widetilde{A}_2 -  \partial_{x_2} \widetilde{A}_1$ and their derivatives until the order $2 |q|$ or $2 |q|-2$:

\bel{defvq}
 {v}_{q} ({\cV}):=\begin{cases} 
\overline{t}_q {\rm L}^{(0)}_{|q|}\left(-\frac{\Delta}{2b}\right) V_{1}  + 
    t_q {\rm L}^{(0)}_{|q|-1}\left(-\frac{\Delta}{2b}\right) V_2 + \frac1{\mu_q}  {\rm L}^{(1)}_{{|q|-1}} \left(-\frac{\Delta}{2b}\right) \widetilde{b}, &q <0\\
    V_1, &q =0\\
t_q {\rm L}^{(0)}_{{q}} 
    \left(-\frac{\Delta}{2b}\right)
     V_{1}  
    + 
   \overline{t}_q {\rm L}^{(0)}_{{q}-1}\left(-\frac{\Delta}{2b}\right) V_2  + \frac1{\mu_q}  {\rm L}^{(1)}_{{q-1}} \left(-\frac{\Delta}{2b}\right) \widetilde{b},  & q >0 \\\end{cases}
\ee

 where $t_{q} := \frac1{2} \Big( 1 +  \frac{m}{|\mu_q|}\Big)$, $\overline{t}_q= 1 - t_q$, 
    $$
      {\rm L}^{(m)}_q(t) : = \sum_{j=0}^q \binom{q+m}{q-j} \frac{(-t)^j}{j!}, \quad t \in \re, \quad q \in \Z_+,  \quad m=0,1
      $$
     are the Laguerre polynomials and $\Delta:= \partial_{x_1}^2 + \partial_{x_2}^2$ denotes the Laplacian. For a more precise statement of the result we refer to Proposition \ref{vqeff} in 
   Section \ref{sinfty}. Let us give some consequences about infiniteness$/$finiteness of the eigenvalues near some Dirac-Landau levels.  The proofs appear  in Section \ref{sinfty}. In the following results, we denote by 
   $\cC_0^{n}$, $n\in \Z_+$, the space of compactly supported functions with  $n$ continuous derivatives on $\R^2$.
	
         \begin{corollary}\label{cor26}
  	Fix $q \in \Z$.  
	Assume that ${\cV}$ defined by \eqref{defV} is such that $V_1 \in \cC_0^{2 |q|}$, $V_2 \in \cC_0^{2 |q|-2}$, $W \in \cC_0^{2 |q|-1}$ and 
	$|V_1| \leq v_1$, $ |V_2| \leq v_2$, {$ |W| \leq w,$}
	for some  non-negative smooth compactly supported functions {$v_1, w \in \cC_0^{2 |q|}$ and $v_2 \in \cC_0^{2 |q|-2}$}.   
	
	If $v_q (\cV) $ defined by \eqref{defvq} satisfies  $v_q (\cV) \neq 0$, then there exists  an infinite number of discrete eigenvalues that accumulate at $\mu_q$. 
   \end{corollary}

       }
       
           \begin{corollary}  \label{prop_inf}
We consider diagonal potentials ${\cV}$ (i.e. $W= 0$ in \eqref{defV}).  
 For any  given  $0\leq V_1 \in \cC_0^{2}(\R^2)$   there exists $V_2 \in \cC_0^{0}(\R^2)$ such that for any $\delta >0$ sufficiently small,
 $D_0+\cV$ has an infinite number of discrete eigenvalues in $(\mu_0, \mu_0+\delta)$  and no eigenvalues in {$(\mu_1, \mu_1+\delta)\cup (\mu_{-1}, \mu_{-1}+\delta) .$} 
      \end{corollary} 
      
   \begin{remark}
  { Under more technical assumptions we can also expect to have potentials for which these different behavior appear near other Landau-Dirac levels.} 
   \end{remark}
}
 
\section{Proof of the main results}\label{s3}

 First, for $T = T^*$ a  self-adjoint compact operator in a Hilbert space and $s>0$, set
 \bel{defn+}
 n_\pm(s;T) = {\rm Tr}\,\one_{(s,\infty)}(\pm T).
 \ee
 That is, $n_+(s,T)$ (resp., $n_-(s,T)$) is the number of the eigenvalues of $T$ counted with the multiplicities, greater than $s>0$ (resp., less than $-s < 0$).

Next, let us recall some  results about the operator $D_0$. The Foldy-Wouthuysen transformation $U_{FW}$ defined by
\begin{equation}\label{ufw}
 U_{FW} = \frac1{\sqrt{2}} \sqrt{I + m | D_{0} |^{-1}} + \sigma_3 \text{ sign}(D_0-m  \sigma_3) \frac1{\sqrt{2}} \sqrt{I - m | D_{0} |^{-1}},
 \end{equation}
 gives us 
$$D_0=
U_{FW}^* \; \left( \begin{array}{cc}  \sqrt{a^*a + m^2  }  &0 \\
0  & - \sqrt{aa^* + m^2 }
\end{array} \right) \; U_{FW}.$$
Notice that  $a^*a = H_L -b $ and  $aa^*= a^*a + 2b = H_L + b $.
In particular, for any $n \in \Z_+$, the following eigenspaces coincide with $ (a^*)^n$ Ker $a$ (see for instance \cite[Section 9]{BrPuRa04}):
\bel{eq1}
\text{Ker} (H_L-\Lambda_n)=\text{Ker} (a^*a - 2bn) = \text{Ker} (aa^* - 2b(n+1))= (a^*)^n \text{Ker}( a).
\ee
Let  $p_n$ be  the orthogonal projection onto   $\text{Ker} (H_L-\Lambda_n)$. Then 
, 
\bel{defPq}
 {\mathcal P}_{q} =\begin{cases}  U_{FW}^* \; \left( \begin{array}{cc} p_q   &0 \\
0  & 0
\end{array} \right) \; U_{FW} & q \in \{0,1,2,...\}\\[1.5em]

 U_{FW}^* \; \left( \begin{array}{cc} 0   &0 \\
0  & p_{|q|-1}
\end{array} \right) \; U_{FW},&q \in \{-1,-2,...\}.
\end{cases}
\ee

\subsection{Reduction to Toeplitz operators}

In this section, for $q \in \Z$ fixed, the analysis of the counting functions $\mathcal N_\pm^q$  (see definition \eqref{defNpm}) is reduced to the spectral study of the Toeplitz operators $\mathcal P_{q}{\cV} \mathcal P_{q}$. Then we connect these operators to Toeplitz operators $p_{|q|} v p_{|q|}$ for some operator $v$ depending on ${\cV}$, $a$ and $a^*$. 

The next result is a consequence  of general consideration that we prove in the Appendix.  The main novelty is that they are   valid   for  not necessarily lower semi-bounded operators {and adapted to matrix operators.}
\begin{proposition}\label{Diag}
	For a given  $\epsilon >0$ define the potentials  
	$$
	{\cV}_\epsilon^\pm :=\begin{pmatrix}
	V_1\pm\epsilon(|V_1|+|W|)&W^*\\
	W&V_2\pm\epsilon(|V_2|+|W|)
\end{pmatrix}  .
	$$  Then, as $\lambda \searrow 0$, we have:
	\begin{align*}
	n_+(\lambda,\mathcal P_q{\cV}_\epsilon^-\mathcal P_q)+O(1)&\leq \mathcal N_+^q(\lambda)\leq  n_+(\lambda,\mathcal P_q{\cV}_\epsilon^+\mathcal P_q)+O(1)
	\\
	n_-(\lambda,\mathcal P_q{\cV}_\epsilon^+\mathcal P_q)+O(1)&\leq \mathcal N_-^q(\lambda)\leq  n_-(\lambda,\mathcal P_q{\cV}_\epsilon^-\mathcal P_q)+O(1).
	\end{align*}
\end{proposition}

The proof appears   in the Appendix, section \ref{ss42}.

\begin{lemma}\label{lem1}
 Let ${\cV}$  be defined by \eqref{defV} with  $V_1, V_2 \in L^{\infty}(\R^2,\R )$ and $W \in L^{\infty}(\R^2,\C )$. Then we have
\begin{align*}
{\mathcal P}_{q} {\cV} {\mathcal P}_{q} =& U_{FW}^* \begin{pmatrix}  T_{q} ({\cV})  & 0  \\
0 & 0
\end{pmatrix}  U_{FW},\quad \text{for}\,\, q\geq 0; \\[0.5em]
{\mathcal P}_{q} {\cV} {\mathcal P}_{q} = &U_{FW}^* \begin{pmatrix}  0  & 0  \\
0 &  T_{q} ({\cV})
\end{pmatrix}  U_{FW},\quad \text{for}\,\, q< 0,
\end{align*}
where 
\begin{align*}
T_{0} ({\cV})  := &  { p_0 V_1 p_0}; \\
T_{q} ({\cV})  := & { t_{q}} { p_q V_1 p_q} + \frac{1- t_q}{2bq} {  p_q a^* V_2 a p_q } + \frac{1}{2 \mu_q}  {  p_q (a^* W +  W^* a) p_q },\quad \text{for}\,\, q> 0; \\
T_{q} ({\cV})   := & { t_{q}} {  p_{|q|-1} V_2 p_{|q|-1} } + \frac{1- t_q}{2b|q|} {  p_{|q|-1} aV_1 a^*  p_{|q|-1} } + \frac{1}{2 \mu_q}   {  p_{|q|-1} ( a W^* + W a^*)  p_{|q|-1} } ,\quad \text{for}\,\, q< 0,
\end{align*}
with 
$t_{q} := \frac1{2} \Big( 1 +  \frac{m}{|\mu_q|}\Big)$.
\end{lemma}
{ \begin{remark}
 Note that thanks to the following relations, for any $n\in \Z_+$, the operators $ap_n$, $a^*p_n$, $p_na^*$ and $p_na$  define bounded operators and then for any $q \in \Z$, the above operators $T_{q} ({\cV})$ are defined without any smoothness assumptions on $V_1$, $V_2$, $W$:
 \begin{align*}\| ap_nf \|^2 & = \langle ap_nf , ap_nf \rangle  & = \langle p_nf , a^*ap_nf \rangle & = \langle p_nf , (H_L-b) p_nf \rangle  = 2bn \| p_nf \|^2\\
\| a^*p_nf \|^2  & = \langle a^*p_nf , a^*p_nf \rangle & = \langle p_nf , aa^*p_nf \rangle & = \langle p_nf , (H_L+b) p_nf \rangle  = 2b(n+1) \| p_nf \|^2.
 \end{align*}
   \end{remark}
}
\begin{proof}
Let us introduce the projections 
$$ \Pi_+ := \begin{pmatrix} 1  & 0  \\
0 & 0
\end{pmatrix}, \qquad  \Pi_- := \begin{pmatrix}  0  & 0  \\
0 & 1
\end{pmatrix}, $$
in such a way that $D_0^2 = (H_L-b+m^2) \Pi_+ + (H_L+b+m^2) \Pi_-$ and  
$${\mathcal P}_{q} = \begin{cases}U_{FW}^* p_{q}  \Pi_+  U_{FW}&q\geq 0,\\
U_{FW}^* p_{|q|-1}  \Pi_-  U_{FW}&q<0.
\end{cases} $$

For $n\in\Z_+$ set  $\nu_n^+={2bn}$ and $\nu_n^-={2b(n+1)}$. 
Then,  using that $(H_L \mp b) p_{n} =  \nu_{n}^\pm p_{n} $, we have 
$$| D_{0} | p_{n}  \Pi_\pm = \sqrt{\nu_{n}^\pm+m^2} p_{n}   \Pi_\pm , \qquad | D_{0} - m \sigma_3 | p_{n}   \Pi_\pm = \sqrt{\nu_n^\pm} p_{n}   \Pi_\pm .$$
Consequently, thanks to \eqref{ufw}  we obtain
\begin{align} \nonumber
U_{FW}^* p_{n}   \Pi_\pm = & \Big(  \frac1{\sqrt{2}} \sqrt{I + \frac{m}{ \sqrt{\nu_{n}^\pm+m^2}}} I -
\frac1{\sqrt{2} \sqrt{\nu_{n}^\pm} } \sqrt{I - \frac{m}{ \sqrt{\nu_{n}^\pm+m^2}}} 
\sigma_3(D_0-m  \sigma_3) \Big)  p_{n}   \Pi_\pm\\ \nonumber
= & \Big( \sqrt{t^\pm_{n}} I - \sqrt{\frac{1-t^\pm_{n}}{
{\nu_{n}^\pm}
}}  \sigma_3(D_0-m  \sigma_3) \Big)  p_{n}   \Pi_\pm   ,
 \end{align}
 with
 $t^\pm_{n} := \frac1{2} \Big( 1 +  \frac{m}{\sqrt{\nu_{n}^\pm+m^2}}\Big)$
 and for the dual:
 \begin{align} \nonumber
p_{n}    \Pi_\pm   U_{FW} =  &   p_{n}  \Pi_\pm    \Big( \sqrt{t^\pm_{n}} I + \sqrt{\frac{1-t^\pm_{n}}{
{\nu_{n}^\pm}
}}  \sigma_3(D_0-m  \sigma_3) \Big) .
 \end{align}
Now, let us first consider the case of ${\cV}$ diagonal (i.e. $W =0$).
Using that $ \sigma_3(D_0-m  \sigma_3)  =  \begin{pmatrix}  0  & a^*  \\
- a & 0
\end{pmatrix}
$, we have 
 \begin{align} \nonumber
\left( \begin{array}{cc}  V_1 & 0  \\
0 & V_2
\end{array} \right) \sigma_3(D_0-m  \sigma_3)  & = \begin{pmatrix}  0  & V_1 a^*  \\
- V_2 a & 0
\end{pmatrix}; \\[1em] \nonumber
- \sigma_3(D_0-m  \sigma_3) \left( \begin{array}{cc}  V_1 & 0  \\
0 & V_2
\end{array} \right) \sigma_3(D_0-m  \sigma_3) & = \begin{pmatrix}  a^* V_2 a  &  0  \\
0  & a V_1 a^*
\end{pmatrix} = a^* V_2 a  \Pi_+  +   a V_1 a^* \Pi_- .
 \end{align}
This  implies that $ \Pi_\pm {\cV} \sigma_3(D_0-m  \sigma_3)   \Pi_\pm = 0 = \Pi_\pm  \sigma_3(D_0-m  \sigma_3)  {\cV}  \Pi_\pm $ and then 
$$  p_{n} \Pi_\pm  U_{FW} {\cV} U_{FW}^* p_{n}   \Pi_\pm = p_{n}\Pi_\pm    \Big( t^\pm_{n}  {\cV} +  \frac{1-t^\pm_{n}}{
{\nu_{n}^\pm}
} ( a^* V_2 a  \Pi_+  +   a V_1 a^* \Pi_- )  \Big) p_{n}   \Pi_\pm .$$
Therefore the expressions for $T_{q}({\cV})$ for ${\cV}$ diagonal, follows from the relations 
$  \Pi_+  {\cV} \Pi_+ = V_1 \Pi_+$, $\Pi_-  {\cV}  \Pi_- = V_2 \Pi_- $ by using that
$$
2b |q|  = \begin{cases} \nu_q^+&q\geq 0\\
\nu^-_{|q|-1} &q<0
\end{cases} ,
\qquad 
t_q = \begin{cases} t_q^+&q\geq 0\\
t^-_{|q|-1} &q<0
\end{cases}.
$$
For $q=0$, the expression of $T_{0}({\cV})$ comes from the fact that $t_{0} = 1$ and   $ap_0=0$. 

For $W \neq 0$, let us now consider the anti-diagonal case (i.e. $V_1=V_2=0$). We have:
 \begin{align} \nonumber
\left( \begin{array}{cc}  0 & W^*  \\
W & 0
\end{array} \right) \begin{pmatrix}  0  & a^*  \\
- a & 0
\end{pmatrix} &   = \begin{pmatrix}  - W^* a   & 0  \\
0 & W a^*
\end{pmatrix}; \\  \nonumber
\begin{pmatrix}  0  & a^*  \\
- a & 0
\end{pmatrix}  \left( \begin{array}{cc}  0 & W^*  \\
W & 0
\end{array} \right)  &  = \begin{pmatrix}   a^* W   & 0  \\
0 & - aW^*
\end{pmatrix} \\ \nonumber
 \end{align}
 and then 
  \begin{align} \nonumber
 \sigma_3(D_0-m  \sigma_3) \left( \begin{array}{cc}  0 & W^*  \\
W & 0
\end{array} \right)   \sigma_3(D_0-m  \sigma_3) & = \begin{pmatrix} 0 &  a^* W a^*  \\
 a W^* a & 0
\end{pmatrix} .
 \end{align}
It implies that for ${\cV}$ anti-diagonal, $ \Pi_\pm   {\cV}   \Pi_\pm  = 0 = \Pi_\pm   \sigma_3(D_0-m  \sigma_3) {\cV}  \sigma_3(D_0-m  \sigma_3)   \Pi_\pm  $ and then 
$$ \Pi_\pm p_{n}  U_{FW}  {\cV}  U_{FW}^* p_n  \Pi_\pm =\sqrt{t^\pm_{n}} \sqrt{\frac{1-t^\pm_{n}}{
{\nu_{n}^\pm}
}} \Pi_\pm p_{n}   \begin{pmatrix}   a^* W +  W^* a & 0  \\
0 & -( aW^* +  W a^* )
\end{pmatrix} p_n  \Pi_\pm .$$
For $q=0$, since $t_{0} = 1$, for the anti-diagonal matrix we have $T_{0}({\cV})=0$.  
Therefore by using that
$$
t^\pm_{n} \frac{1-t^\pm_{n}}{\nu_{n}^\pm} = \frac1{4 \nu_{n}^\pm} 
\Big( 1 -  \frac{m^2}{\nu_{n}^\pm+m^2}\Big)=  \frac{1}{4(\nu_{n}^\pm+m^2)}
$$
and that 
$$
 \mu_q = \begin{cases} \sqrt{\nu_{q}^+ + m^2} &q\geq 0\\
- \sqrt{\nu_{|q|-1}^- + m^2}  &q<0
\end{cases} ,
$$
the expressions for $T_{q}({\cV})$ for $q \neq 0$, follow by summing up the results for diagonal and anti-diagonal ${\cV}$. 
\end{proof}

\begin{remark}
The results above are also valid in the massless case $m=0$. In this scenario, the primary simplification is that $t_q=\frac12$, which will not significantly influence  the  spectral analysis. 
\end{remark}


\subsection{Eigenvalue distribution for the Toeplitz operators ${\mathcal P}_{q} {\cV} {\mathcal P}_{q}$}

As Proposition \ref{Diag} shows, the distribution of the eigenvalues of $D_0 + {\cV}$ near each $\mu_{q}$, $q \in \Z$ is  governed by the Toeplitz operators ${\mathcal P}_{q} {\cV}  {\mathcal P}_{q}$. This is  the  usual case for the  perturbation of isolated eigenvalues of infinite multiplicities.  Denote by $\left\{\nu_{k,q}^\pm({\cV})\right\}_{k \in \Z_+}$ the non-increasing (resp. non-decreasing) sequence of  positive (resp. negative) eigenvalues of ${\mathcal P}_{q} {\cV}  {\mathcal P}_{q}$. We denote  by ${\rm ess\,supp} \, {\cV}$ the union of the essential supports of $V_1$, $V_2$ {and $W$}.

\begin{proposition}\label{thm2}
	 	Under  conditions of Theorem \ref{thm1} (resp. of Theorem  \ref{theo1_b}) for any $q \in \Z^*$, we have
	\bel{15mai24}
	\liminf_{k \to \infty} \left(k! \nu_{k,q}^+({\cV})\right)^{1/k} \geq \frac{b\, \capa(K_-(\chi_\Omega w))^2}{2}	,
	\qquad
\limsup_{k \to \infty} \left(k! \nu_{k,q}^+({\cV})\right)^{1/k} \leq \frac{b\, \capa({\cK} )^2}{2}.
	\ee

	At the same time   
	\bel{6mai2024}
	\#\{\nu_{k,q}^-({\cV}) \}<\infty.
	\ee  
	For $q=0$, the above estimates are still true except that for the lower bound, under  conditions of  Theorem  \ref{theo1_b}, we have to replace $\chi_\Omega w$ by $\chi_{\Omega_1} w$ (and 
 in the upper bound, $\cK$ can be replaced by $\cK_1$, a compact set with Lipschitz boundary which contains ${\rm ess\,supp} \, V_1 \cup {\rm ess\,supp} \, W$).
\end{proposition}

In the next part of this  section we prove the previous proposition.   By combining Lemma \ref{lem1} with the estimates on the quadratic forms, we reduce the study of ${\mathcal P}_{q} V  {\mathcal P}_{q}$ to $T_{q} ({\cV})$ and then to Toeplitz operators of the form $p_0vp_0$. 
As used in previous works (see e.g. \cite{AhCa79,FiPu06}), we will exploits complex analysis tools. To {$(x_1,x_2) \in \R^2$  there} corresponds $z = x_1 + i x_2 \in \C$ and the differential operators $a$ and $a^*$ defined by \eqref{defa} are connected to the complex derivatives:
$$
\partial:= \partial_{x_1} - i \partial_{x_2}, \qquad \overline{\partial}:= \partial_{x_1} + i \partial_{x_2},
$$
by the relations:
$$
a = -i \left(\overline{\partial} + \frac{zb}2 \right)= -i e^{-b|z|^2/4} \overline{\partial} e^{b|z|^2/4}; \qquad 
a^* = -i \left({\partial} - \frac{\overline{z} b}2\right)= -i e^{b|z|^2/4} {\partial} e^{-b|z|^2/4}.
$$
Thus 
${\text Ker} (H_L-b)= {\text Ker} (a^*a) = {\text Ker} (a)$ is given by the functions $u \in L^2(\R^2)$ such that $f: z \mapsto e^{b|z|^2/4} u(z) $ is entire.
Then we consider the Fock space $\mathcal F$ of entire functions $f$ such that
$$
\|f\|_\mathcal{F}^2=\int_{\C}|f(z)|^2 e^{-b|z|^2/2}dm(z)<\infty ,
$$
in such a way that ${\text Ker} (H_L-b) = e^{-b|z|^2/4} \mathcal F  $ and 
$${\text Ker} (H_L-\Lambda_n)= {\text Ker} (a^*a - 2bn) = (a^*)^n {\text Ker}( a) = (a^*)^n \left( e^{-b|z|^2/4} \mathcal F  \right) . $$


Then,  using that for $n \in \N$
$$a(a^*)^n p_0 = a a^* p_{n-1} = (a^*a + 2b) p_{n-1} = 2nb (a^*)^{n-1} p_{0}, \qquad a^n(a^*)^n p_0 = (2b)^n n! p_0,$$
we deduce that for ${\mathcal H}_n:= {\text Ker} (H_L-\Lambda_n)$, the operator 
$$ \sqrt{(2b)^n n! } \, (a^*)^n \; : {\mathcal H}_0 \longrightarrow {\mathcal H}_n,$$
is an isometry onto ${\mathcal H}_n$. Moreover, for 
$p_n$ the orthogonal projection onto  ${\mathcal H}_n$, and a function $v \in L^\infty(\R^2,\R)$, the quadratic form of the operator ${p_n v p_n}_{| {\mathcal H}_n}$ is unitarily equivalent to $\mathfrak a_n(v)$ defined on $\mathcal F$ by
\begin{align}
	\mathfrak a_n(v)[f]&:=\frac1{(2b)^n n!}\int_{\C} |(a^*)^n e^{-b|z|^2/4} f(z)|^2 v(z)dm(z),
\end{align}
and exploiting that $(a^*)^n (e^{-b|\cdot |^2/4} f)(z)= e^{-b|z|^2/4}  (a^*)^n (f )(z)$, we have
\begin{align}\label{defan}
	 \mathfrak a_n(v)[f]&=\frac1{(2b)^n n!} \int_{\C} |(a^*)^n  f(z)|^2 e^{-b|z|^2/2} v(z)dm(z)  \\
	&=\frac1{(2b)^n n!} \int|(2\partial-b\overline z)^nf|^2v(z)dm(z).
\end{align}
We also obtain that ${(2bn)^{-1} p_n a^* v  a p_n}_{| {\mathcal H}_n}$ (resp. ${(2bn)^{-1} p_n a v  a^* p_n}_{| {\mathcal H}_n}$) 
is unitarily equivalent to $\mathfrak a_{n-1}(v)$  (resp. $\mathfrak a_{n+1}(v)$). 

At last, for $w \in  L^\infty(\R^2,\C)$, the quadratic form of the operator ${\frac12 p_n (a^* w + w^* a) p_n}_{| {\mathcal H}_n}$
 (resp.  ${\frac12 p_n (a w^* + w a^*) p_n}_{| {\mathcal H}_n}$  ) 
 is unitarily equivalent to $\mathfrak b_n(w)$ (resp. $\mathfrak b_{n+1}(w)$) with $\mathfrak b_n$ defined on $\mathcal F$ by
\begin{align}\label{defbn}
	 \mathfrak b_n(w)[f]&= \frac1{(2b)^{(n-1)} (n-1)!}   {\rm Re}  \int_{\C} (a^*)^n  f (z) \overline{(a^*)^{n-1}  f (z)} e^{-b|z|^2/2} w(z) dm(z).
\end{align}

%

For $T_{q} ({\cV})$, $q \in \Z$, introduced in Lemma \ref{lem1}, we deduce:

\begin{lemma}\label{lem1b}
 Let ${\cV}$  be defined by \eqref{defV} with  $V_1, V_2 \in L^{\infty}(\R^2,\R )$ and $W \in L^{\infty}(\R^2,\C )$. 
 For $q \in \Z$, the operator $T_{q} ({\cV})$,  introduced in Lemma \ref{lem1}, is unitarily equivalent to the operator in $\mathcal F$ given by the quadratic form 
 \begin{align*}
 a_0(V_1) &  & \text{ for } q=0 \\
 t_q \mathfrak a_q(V_1) & + (1-t_q) \mathfrak a_{q-1}({V_2}) + \frac1{\mu_q} \mathfrak b_q({W})  & \text{ for } q\geq 1 \\
 (1-t_q) \mathfrak b_{|q|}({V_1}) & + t_q \mathfrak a_{|q|-1}({V_2}) + \frac1{\mu_q}  \mathfrak b_{|q|} ({W}) & \text{ for } q \leq -1 
 \end{align*}
with 
$t_{q} := \frac1{2} \Big( 1 +  \frac{m}{|\mu_q|}\Big)$.
\end{lemma}

%


Now let us analyze, theses quadratic forms.

\begin{lemma}\label{lemTq}
Let $V_1$, $V_2$, $W$ be compactly supported functions in $L^\infty(\R^2)$ and $q \geq 1$. 

Under the assumptions of Theorem \ref{thm1} (resp. of Theorem \ref{theo1_b}), there exist a constant $C_1>0$ and a subspace of $\mathcal F$ of finite codimension, on which we have:
\bel{lemlower}
 \mathfrak a_q(V_1) + \mathfrak a_{q-1}(V_2) + { \mathfrak b_q(W) }   \geq C_1 \mathfrak a_0(\chi_\Omega w).
 \ee

\end{lemma}  
\begin{proof}
We assume for simplicity  $b=2$. 
{ 
From the Cauchy-Schwarz inequality we have the estimate, for any $\delta >0$:
\bel{CSba}
| \mathfrak b_q(W)[f] | \leq \delta \mathfrak a_q(|W|)[f] + \frac{1}{\delta} \mathfrak a_{q-1}(|W|)[f].
\ee
}
So, under the assumption  { $V_1 \geq \chi_\Omega w$, $ V_{2} \geq - C_2 \chi_\Omega w $ and $|W| \leq C_3 \chi_\Omega w$, there exist $C_1>0$ and $\widetilde{C}_2>0$ such that for any $f \in \mathcal F$, }
\begin{align}
\begin{split}
\mathfrak a_q(V_1)[f] + & \mathfrak a_{q-1}(V_2)[f] + { \mathfrak b_q(W)[f] }\geq  \\ \nonumber
& C_1\int_{\Omega} |(\partial -\overline z)^q  f(z)|^2 w(z)dm(z)
	-\widetilde{C}_2 \int_{\Omega} |(\partial -\overline z)^{(q-1)}  f(z)|^2 w(z)dm(z).
	\end{split}
\end{align}

Since $w$ is a  weight for $\Omega$,  $H^1(\Omega;w)$ is compactly embedded in $L^2(\Omega;w)$, and therefore for any $\gamma>0$ there exists a subspace  of finite codimension in $H^1(\Omega;w)$ such that
$
\|f\|_{L^2(\Omega;w)}\leq \gamma\|\nabla f\|_{L^2(\Omega;w)}
$. 
 Since {for $f \in \mathcal F$  we have that 
for any $k \in \Z_+$, $\partial^k f_{|\Omega}  \in H^1(\Omega;w)$}, 
\Bk
there exists a subspace $N\subset \mathcal{F}$ of a finite codimension $l$ such that 
\bel{eq0}
\|\partial^{k} f\|_{L^2(\Omega;w)}\leq \gamma \|\partial^{k+1} f\|_{L^2(\Omega;w)}\qquad f\in N, \quad k=0,1,2,...,q-1.
\ee
 Then, for any $f\in N$ the contributions of the derivatives of order  less than $q$ are controlled by the $q$-th derivative: \Bk
$$\mathfrak a_q(V_1)[f] + \mathfrak a_{q-1}(V_2)[f] + { \mathfrak b_q(W)[f] } \geq C_1  \|\partial ^qf\|^2_{L^2(\Omega;w)}  - C \sum_{k=1}^{q} \|\partial^{q-k}f\|_{L^2(\Omega;w)}^2
	\geq C_\gamma  \|f\|^2_{L^2(\Omega;w)}$$
with some $C, C_\gamma >0$ if we take $\gamma $ sufficiently small. This implies (\ref{lemlower}) when ${\cV}$ satisfies the assumptions of Theorem \ref{thm1}. 

When $V_i$ , $i=1,2$, satisfy  $V_i \geq \chi_{\Omega_i} w - C \chi_K$, by linearity and monotonicity of $ V \mapsto \mathfrak a_q(V)$ we have: 
\begin{equation}\label{eq1a}
\mathfrak a_q(V_1)+  \mathfrak a_{q-1}(V_2) \geq \mathfrak a_q(\chi_{\Omega_1} w) +\mathfrak a_{q-1}(\chi_{\Omega_2} w)- C \mathfrak a_q(\chi_K) - C  \mathfrak a_{q-1}(\chi_K). 
 \end{equation} 
 {Moreover using again \eqref{CSba}, for $|W | \leq C_3 \chi_{\Omega_1} w + C_4 \chi_K$ we have 
 \begin{equation}\label{eq1b}
\mathfrak a_q(V_1)+ \mathfrak a_{q-1}(V_2) +  \mathfrak b_q(W) \geq \frac12 \mathfrak a_q(\chi_{\Omega_1} w) +\mathfrak a_{q-1}(\chi_{\Omega_2} w)- C \mathfrak a_q(\chi_K) - C\mathfrak a_{q-1}(\chi_K), 
 \end{equation} 
 with another $C>0$.}
 
Now we show that the encirclement property allows to absorb the negative contributions by following  
ideas of the proof of Theorem 1.1 of \cite{PuRo11}. First, since $w$ (resp. $w_{|\Omega_1}$) is a weight in $\Omega$ (resp. $\Omega_1$), as above, for $q \geq 1$, on a subspace  of finite codimension in $\mathcal F$, we have 
\begin{equation}\label{eq2}
\mathfrak a_q(\chi_{\Omega_1} w) +\mathfrak a_{q-1}(\chi_{\Omega_2} w) \geq 
C_1\mathfrak a_{q-1}(\chi_{\Omega_1} w) +  \mathfrak a_{q-1}(\chi_{\Omega_2} w) \geq
C_0  \mathfrak a_{q-1}(\chi_{\Omega} w) \geq
c_0 \mathfrak a_{0}(\chi_{\Omega}w),
 \end{equation} 
 for some positive constants $C_1$, $C_0$, $c_0$, where we exploit that $\Omega_1 \cup \Omega_2 = \Omega$ and then $\chi_{\Omega_1} + \chi_{\Omega_2} \geq \chi_{\Omega}$.

 By assumption, there exist a Jordan curve $\Gamma$ contained in { $\Omega$, whose} interior contains
 $K$. Let $\mathcal U_i$ be the interior part of $\Gamma$, and let  $\mathcal U_e$ be the intersection of the exterior of $\Gamma$ with $\Omega$.
 
Then, thanks to the Cauchy formula, for any $k= 1, \cdots, q$, the operator $\partial^k:{ \mathcal F}\cap L^2(\Gamma, dl)\to L^2(K, dm)$ has    an integral kernel $K_k(z, \zeta)$  smooth on $K \times \Gamma$ such that
 $$\partial^k f(z) = \int_\Gamma K_k(z,\zeta) f (\zeta) dl(\zeta), \quad z \in K \subset \mathcal U_i $$
 where $dl(\zeta)$ is the length measure on the curve $\Gamma$. Thus, this  integral operator is compact (with smooth bounded integral kernels), then for any 
 $\varepsilon >0$, there exists $\mathcal L_\varepsilon \subset \mathcal{F}$, a finite dimensional subspace of $\mathcal{F}$, on which 
\begin{equation}\label{eq3}
\mathfrak a_q(\chi_K)[f] +\mathfrak  a_{q-1}(\chi_K)[f] \leq \varepsilon \int_\Gamma  | f (\zeta) |^2  dl(\zeta) \leq \frac{\varepsilon}{  \eta} \int_{\Omega}  | f (\zeta) |^2 w(\zeta) dm(\zeta),
 \end{equation} 
 with $\eta > 0 $ depending on   the weight $w$ and  the thickness of a  neighborhood of $\Gamma$ included in the open set $\Omega$ (we also have used that $w$ admits a positive lower bound on this set).
 
 By choosing $\varepsilon$ sufficiently small, combining inequalities \eqref{eq1b}, \eqref{eq2} and \eqref{eq3}, we deduce (\ref{lemlower}) under the assumption of Theorem \ref{theo1_b}.
 
\end{proof}

{
\begin{remark}\label{rmqV1+}
This  proof exploits the result of Lemma \ref{lem1} and the expression of $T_{q} ({\cV})$. It shows that {when $V_1$ and $V_2$ have the same support,} it is the contribution of $p_q V_1 p_q$ that dominates and compensates for any negative contribution of $V_2$. This occurs despite the fact that the behavior of the eigenvalue counting functions of $p_q V_1 p_q$ and $p_q a^* V_2 a p_q$ are similar. {The assumptions are also made such that the contribution of $W$ is absorbed by $p_q V_1 p_q$  (see \eqref{CSba}). \Bk}
\end{remark}
}

      \subsection{Proof of Proposition \ref{thm2} 
      } 
\label{s4}
\begin{proof} {Thanks to Lemma \ref{lem1}, the non-zero eigenvalues of ${\mathcal P}_{q} {\cV}  {\mathcal P}_{q}$ are those of $T_{q} ({\cV})$. For $q=0$, $T_0({\cV}) = p_0 V_1 p_0$ then the result is a direct consequence of { \cite[Lemmas 1 and 2]{FiPu06}}. For $q \neq 0$, the proof is a modification of these lemmas.}  By applying Lemma \ref{lem1b}, we are reduced to study the operator, on the Fock space $\mathcal F$, given by the quadratic form
$$\mathfrak a_{|q|}({V_1}) +\mathfrak a_{|q|-1}({V_2}) + \frac1{\mu_q} \mathfrak b_{|q|} ({W}),$$
with $\mathfrak a_n(V)$ and $\mathfrak b_q(W)$ defined by \eqref{defan} and   \eqref{defbn} (here we include the positive coefficients $t_q$ and $(1-t_q)$ in $V_1$ or $V_2$).
%
%
Let us start with the lower bound. Thanks to Lemma \ref{lemTq}, on $N \subset \mathcal F$ of finite codimension, we have
$$\mathfrak a_{|q|}({V_1})[f]  +\mathfrak a_{|q|-1}({V_2})[f] + \frac1{\mu_q}  \mathfrak b_q({W})[f] \geq C_1 \int_{\Omega} | f(z)|^2 w(z)dm(z).$$

Then the number of negative eigenvalues is at most finite, which implies in particular that \eqref{6mai2024} holds. Moreover as in the proof of the lower bound in { \cite[Lemma 1]{FiPu06}} we obtain that for each $\varepsilon>0$ and $k$ large enough
$$
\left(k! \nu_{k,q}^+(V)\right)^{1/k}\geq (1+\varepsilon)^2 M_k(w)^{1/k},
$$
where, for ${\bf P}_k$ the set of all monic polynomials in $z$ of degree $k \in \N$,
$$
M_k(w):=\inf_{p \in {\bf P}_k}\int_{\C} |p(z)|^2 w(z)dm(z).
$$ 
Now, to finish the proof of the lower bound we notice that 
$$
\liminf_{k \to \infty}M_k(w)^{1/k}\geq
\frac{b\, \capa(K_-(w))^2}{2}
$$
(see (1.15)-(1.16) in \cite{FiPu06}).

{For the proof of the upper bound we use that a.e. $V_1,  V_2, |W| \leq C\chi_{ \cK}$ and 
we combine \eqref{CSba} with the Cauchy integral formula. Let $\cK_\delta:= \{ z \in \C; \; {\rm dist} (z, \cK) \leq \delta \}$, $\delta >0$.  For any $f \in \mathcal F$ and $ k \in \{0, \cdots , q\}$, there exist a positive constant $C_k$ such that 
$$\sup_{\cK} | \partial^k f | \leq \frac{C_k}{\delta^k} \sup_{\cK_\delta} | f |,$$
and there exists $C_\delta >0$ such that 
$$\mathfrak a_{|q|}({V_1})[f]  +\mathfrak a_{|q|-1}({V_2})[f] + \frac1{\mu_q}  \mathfrak b_q({W})[f] \leq C_\delta \int_{\cK_\delta} | f(z)|^2 dm(z).$$
Then following the proof of the upper bound  in  {\cite[Lemma 2]{FiPu06}}, we deduce 
$$ \limsup_{k \to \infty} \left(k! \nu_{k,q}^+({\cV})\right)^{1/k} \leq \frac{b\, \capa({\cK_\delta} )^2}{2},$$
and we conclude by using that for the Lipschitz set $\cK$,  $\capa({\cK_\delta} )$ tends to $\capa({\cK} )$ as $\delta \searrow 0$. }


\end{proof}

\subsection{Proof of Theorem \ref{thm1} and Theorem \ref{theo1_b}}\label{s42bis}
Starting from Proposition  \ref{Diag} we are reduced to  the study   of   the eigenvalues of  $\mathcal P_q {\cV}_\epsilon^\pm \,\mathcal P_q$ 
where 
$$
	{\cV}_\epsilon^\pm :=\begin{pmatrix}
	V_1\pm\epsilon(|V_1|+|W|)&W^*\\
	W&V_2\pm\epsilon(|V_2|+|W|)
\end{pmatrix}  .
	$$  
Thus the proof consists first to check that the assumptions on ${\cV}$ are still true for ${\cV}_\epsilon^\pm$ with $\epsilon >0$ sufficiently small. { First we obviously have
${\rm ess\,supp} \, {\cV}_\epsilon^\pm \subset {\rm ess\,supp} \, {\cV} \subset \cK.$
Then under the condition of Theorem \ref{thm1}, $V_1$ is non-negative then   $V_1+\epsilon(|V_1|+|W|)\geq  (1+\epsilon) V_1 \geq \chi_\Omega w (1+\epsilon)$. On the other side,  
the condition $|W|\leq C_3\chi_\Omega \omega$
 implies  
$$
V_1-\epsilon(|V_1|+|W|)\geq  (1-\epsilon) V_1- C_3\chi_\Omega \omega\geq \chi_\Omega w (1-\epsilon(1+ C_3)).
$$ \Bk
A lower bound  of  $V_2\pm\epsilon(|V_2|+|W|)$ \Bk by $-C_\epsilon \chi_\Omega w $ still holds for some $C_\epsilon \in \R$.

\Bk
Concerning the assumptions of Theorem  \ref{theo1_b}, without loss of generality we can assume that $\Omega \subset \R^2\setminus K$. Consequently if $V_1$ and $V_2$ are bounded and, on  $\Omega$, satisfy $V_i \geq  \chi_{\Omega_i} w$, $i=1,2$, then on  $\Omega$, $V_i \pm\epsilon|V_i| = (1 \pm \epsilon) V_i \geq (1 \pm \epsilon) \chi_{\Omega_i} w $ and ${\cV}_\epsilon^\pm $ still satisfy (\ref{V=}). 

Then we can apply Proposition \ref{thm2}  to  ${\cV}_\epsilon^\pm$ and  obtain the corresponding asymptotics \eqref{15mai24},  \eqref{6mai2024} for  the eigenvalues of  $\mathcal P_q {\cV}_\epsilon^\pm \,\mathcal P_q$.
Now,  following the proof of \cite[Corollary 5.11]{BrRa20} we deduce the asymptotic expansions of the counting functions of Theorem \ref{thm1} and  Theorem \ref{theo1_b}. Finally, thanks to Proposition \ref{Diag}, we treat the case when $-{\cV}$ satisfies (\ref{V1>}) or (\ref{V=}) by interchanging ${\mathcal N}_+^q(\lambda)$ and ${\mathcal N}_-^q(\lambda)$ because $n_-(s; T) = n_+(s, -T)$.

\subsection{Reduction to Toeplitz operators $p_0 v p_0$ for smooth potentials
\label{sinfty}}

For a smooth potential ${\cV}$, let us give a direct consequence of  {\cite[Appendix A.1]{BrRa20}}, which allows us to {pass from the study of  the  Toeplitz operators $\mathcal P_{q}{\cV}   \mathcal P_{q}$ to the study of the  Toeplitz operator $p_0 v_q ({\cV}  )p_0$.} 
From this reduction, we will deduce {Corollary }\ref{cor26} and {Corollary }\ref{prop_inf}.

\begin{lemma}[\cite{BrRa20}]\label{lem2}
Let $q \in \Z^*$ and $T_{q} ({\cV})$ be defined in Lemma \ref{lem1}. Assume that the functions $V_1$, $V_2$ and $W$ are 
smooth compactly supported functions.
Then there exists an unitary operator ${\mathcal W_q}$ such that :
$$T_{q}({\cV})  =  {\mathcal W_q}^*  p_0 v_q({\cV})  p_0 {\mathcal W_q} ,$$
with $v_q({\cV})$ being the smooth function given {by \eqref{defvq}}.
\end{lemma}
\begin{proof}
For
       \bel{au73}
       \varphi_{k,0}(x) : = \sqrt{\frac{b}{2\pi k!}} \left(\frac{b}{2}\right)^{k/2} z^k e^{-b|x|^2/4}, \quad x \in \rd, \quad z = x_1 + i x_2, \quad k \in \Z_+,
       \ee
        \bel{au74}
       \varphi_{k,n}(x) : = \sqrt{\frac{1}{(2b)^n n!}}(a^*)^n\varphi_{k,0}(x), \quad x \in \rd, \quad k \in \Z_+, \quad n \in  \N,
       \ee
      the family $\left\{\varphi_{k,n}\right\}_{k \in \Z_+}$ is an orthonormal basis of $p_n L^2(\rd)$ called sometimes {\em the angular momentum basis}
       (see e.g. \cite{RaWa02} or \cite[Subsection 9.1]{BrPuRa04}). Exploiting that for $k \in \Z_+$ it satisfies
       \bel{37}
       a^* \varphi_{k,n} = \sqrt{2b(n+1)}  \varphi_{k,n+1}, \quad n \in \Z_+,
       \quad a \varphi_{k,n} = \left\{
       \begin{array} {l}
       \sqrt{2bn}  \varphi_{k,n-1}, \quad n \geq 1,\\
       0 , \quad n = 0,
       \end{array}
       \right.
       \ee
the following property is proved in \cite[Appendix A.1]{BrRa20}:
For $
    g_0 = I$,  $g_1 = a^*$, $ g_2 = a$,
     and $\{ w_{jk}\}_{j,k \in \{0,1,2\}}$ smooth bounded functions such that $w_{jk}^* = w_{kj}$, the operator
     $$p_n \left(\sum_{j,k=0}^2 g_j^* w_{jk} g_k\right) p_n
    $$
    with domain $p_n L^2(\re^2)$ is unitarily equivalent to the operator $ p_0 \upsilon_n p_0$, self-adjoint on its domain $p_0 L^2(\rd)$, 
     where for $n \geq 1$, $\upsilon_n : \rd \to \re$ is the bounded multiplier given by
     \begin{align} \label{f14}
     \upsilon_n &: = {\rm L}_{n}\left(-\frac{\Delta}{2b}\right) w_{00} + 2b(n+1) {\rm L}_{n+1}\left(-\frac{\Delta}{2b}\right) w_{11} + 2bn {\rm L}_{n-1}\left(-\frac{\Delta}{2b}\right) w_{22}\\[3mm] \nonumber
    &- 8 {\rm Re}\,{\rm L}_{n-1}^{(2)}\left(-\frac{\Delta}{2b}\right){\partial_z^2 w_{21}}
    - 4 {\rm Im}\,{\rm L}_{n}^{(1)}\left(-\frac{\Delta}{2b}\right){\partial_z w_{01}}
    - 4 {\rm Im}\,{\rm L}_{n-1}^{(1)}\left(-\frac{\Delta}{2b}\right){\partial_z w_{20}},
    \end{align}
    where
 ${\rm L}_n := {\rm L}_n^{(0)} $. 

By applying this result to $n=q>0$ and to
$$ w_{00} = t_q V_1, \quad w_{22} = \frac{1-t_q}{2bq} V_2 , \quad w_{20} = \frac1{2\mu_q} W , \quad w_{11}=w_{21}=w_{01}=0,$$
we obtain the claimed result for $q>0$. For $q<0$, we take $n=|q|-1$ and 
$$ w_{00} = t_q V_2, \quad w_{11} = \frac{1-t_q}{2b|q|} V_1 , \quad w_{01} = \frac1{2\mu_q} W , \quad w_{22}=w_{21}=w_{20}=0.$$
We also use that ${\rm Im} \partial_z W = -\frac{ 1 }2 {\widetilde b}$. 
 \end{proof}

   By combination of Proposition \ref{Diag},  Lemma \ref{lem1} and Lemma \ref{lem2} we have:

   \begin{proposition}\label{vqeff}
	Fix $q \in \Z$ and $\epsilon >0$. Assume that ${\cV}$ defined by \eqref{defV} is such that :
	{
	\begin{enumerate}
	\item 
	$V_1 \in \cC_0^{2 |q|}$, $V_2 \in \cC_0^{2 |q|-2}$ and $W \in \cC_0^{2 |q|-1}$
	\item There exists $v_1, w \in \cC_0^{2 |q|}$ and $v_2 \in \cC_0^{2 |q|-2}$ non-negative smooth compactly supported functions such that 
	$|V_1| \leq v_1$, $ |V_2| \leq v_2$, $|W| \leq w.$
	\end{enumerate}
	  Then for
	  \bel{5dec24}
	{\cV}_{\text{reg},\epsilon}^\pm :=\begin{pmatrix}
	V_1\pm \epsilon (v_1 +w) &W^*\\
	W&V_2\pm\epsilon(v_2+w)
\end{pmatrix}  \ee
}
 and $v_q (\cdot) $ defined by \eqref{defvq}, as $\lambda \searrow 0$, we have:
	$$
	n_+(\lambda, p_0 v_q({\cV}_{\text{reg},\epsilon}^-) p_0)+O(1)\leq \mathcal N_+^q(\lambda)\leq  n_+(\lambda,p_0 v_q({\cV}_{\text{reg},\epsilon}^+) p_0)+O(1),
	$$
	$$
	n_-(\lambda,p_0 v_q({\cV}_{\text{reg},\epsilon}^+) p_0)+O(1)\leq \mathcal N_-^q(\lambda)\leq  n_-(\lambda,p_0 v_q({\cV}_{\text{reg},\epsilon}^-) p_0)+O(1),
	$$
	where $n_\pm(\lambda,\cdot)$ denotes the counting function defined by \eqref{defn+}.
	   \end{proposition}

 Let us mention that the assumptions of  Proposition \ref{vqeff}  are easily satisfied if the smooth coefficients $V_1$, $V_2$ are of definite sign (e.g. if $V_1 \geq$ and $V_2\leq 0$, we choose $v_1=V_1$ and $v_2 = -V_2$) and if $W= \gamma W_0$ with $\gamma \in \C$ and $W_0$ smooth of definite sign. 

 \begin{proof} 
From Proposition \ref{Diag}, it is sufficient to estimate the counting functions of the eigenvalues of  $\mathcal P_{q}{\cV}_{\epsilon}^\pm \mathcal P_{q}$ with ${\cV}_{\epsilon}^\pm$ defined in Proposition \ref{Diag}. By \emph{(2)}, we have that in the sense of quadratic forms  
$${\cV}_{\epsilon}^+ \leq {\cV}_{\text{reg},\epsilon}^+ \quad , \qquad  {\cV}_{\text{reg},\epsilon}^- \leq {\cV}_{\epsilon}^-.$$
Then, from the Min-Max Principle we deduce, for $\lambda>0$:
$$
n_\pm (\lambda, \mathcal P_{q}{\cV}_{\epsilon}^\pm \mathcal P_{q} ) \leq n_\pm (\lambda, \mathcal P_{q}{\cV}_{\text{reg},\epsilon}^\pm \mathcal P_{q} ) \quad , \qquad 
n_\pm (\lambda, \mathcal P_{q}{\cV}_{\text{reg},\epsilon}^\mp \mathcal P_{q} ) \leq n_\pm (\lambda, \mathcal P_{q}{\cV}_{\epsilon}^\mp \mathcal P_{q} ).
$$
Moreover, thanks to Lemma \ref{lem1} and Lemma \ref{lem2} for $\bullet =$"$+$ or "$-$", we have
$$n_\bullet (\lambda, \mathcal P_{q}{\cV}_{\text{reg},\epsilon}^\pm \mathcal P_{q} )= n_\bullet (\lambda, T_{q}({\cV}_{\text{reg},\epsilon}^\pm))=n_\bullet (\lambda, p_0v_{q}(\cV_{\text{reg},\epsilon}^\pm)p_0).$$
We conclude by using these estimates in Proposition \ref{Diag}. 
 \end{proof}
 
 }

   \begin{proof} \emph {of Corollary  \ref{cor26}:} Thanks to     \cite[Theorem A]{Lu08}, if $v_q ({\cV}) \neq 0$, then $p_0 v_q ({\cV})  p_0$ admits an infinite number of discrete eigenvalues that accumulate at $0$. Then 
   from Proposition \ref{vqeff}, it is sufficient to prove that $v_q ({\cV}_{\text{reg},\epsilon}^\pm) \neq 0$ for some $\epsilon \in (0,1)$ with 
   ${\cV}_{\text{reg},\epsilon}^\pm = {\cV} \pm \epsilon \cV$, 
   $
\cV = \begin{pmatrix}
	v_1+w &0\\
	0&v_2 + w
\end{pmatrix}$. 
Since $v_q ({\cV}_{\text{reg},\epsilon}^\pm) = v_q ({\cV}) \pm \epsilon v_q(\cV)$ 
and $v_q(\cV)$ is bounded, then for $\epsilon >0$ sufficiently small, $v_q ({\cV}_{\text{reg},\epsilon}^\pm) \neq 0$ when $v_q ({\cV}) \neq 0$.
       \end{proof}

 \begin{proof} \emph {of Corollary  \ref{prop_inf}:} First, let us recall that $L_0(t)=1$ and $L_1(t)=1-t$, then $v_{\pm 1}({\cV})$ defined by \eqref{defvq} are given by 
 \beas v_1({\cV}) & =  t_1 ( V_{1}+\frac1{2b}\Delta V_{1}) + (1-t_1)V_{2} + \frac1{\sqrt{2b + m^2}}  \widetilde{b}, \\
 v_{-1}({\cV}) & = (1- t_1) ( V_{1}+\frac1{2b}\Delta V_{1}) + t_1 V_{2} - \frac1{\sqrt{2b + m^2}}  \widetilde{b},
 \eeas
 and $ v_{0}({\cV})=V_1$. 
 Thus for $W=0$ (and then $ \widetilde{b} =0$), the effective potentials $v_1({\cV})$ and $v_{-1}({\cV})$ are almost the same and we only prove the result near $\mu_1$.
 Thanks to Proposition \ref{vqeff} and \cite[Theorem A]{Lu08} in order to prove Corollary  \ref{prop_inf}, it is sufficient to choose $V_1$ and $V_2$ such that for some $\epsilon >0$, 
$ v_{0}({\cV}_{\text{reg},\epsilon}^-)= V_1 - \epsilon |V_1| \geq 0$ and $v_1({\cV}_{\text{reg},\epsilon}^+)\leq 0$ (non zero on a non-empty open set). We take  $ V_1 \geq 0$ (positive on a non-empty open set) and $V_2 \leq 0$ such that
 \bel{ineg}
 \Big( t_1 ( {\cV}^+_{1,\epsilon}+\frac1{2b}\Delta V^+_{1,\epsilon}) + (1-t_1)V^+_{2,\epsilon} \Big) \leq 0, \qquad V^+_{i,\epsilon} := V_i + \epsilon | V_i |, \quad i=1,2 .
\ee 
Since $ V_1 \geq 0$, we have $V^+_{1,\epsilon}= (1 + \epsilon) V_1$, and  for 
$$V_2:= -  \frac{2t_1}{1-t_1} \vert V_{1}+\frac1{2b}\Delta V_1  \vert ,$$
we have $V^+_{2,\epsilon}= (1 - \epsilon) V_2$. Consequently \eqref{ineg} is satisfied for any $\epsilon>0$ such that $\frac{1+\epsilon}{1-\epsilon}  < 2$, and the result follows.

   \end{proof}

\section{Appendix: Proof of Proposition \ref{Diag} and the encirclement property}

\subsection{The index of a pair of projections } 
In this section, we give the proof of Proposition \ref{Diag}, following the lines of the the proof of \cite[Theorem 4.1]{PuRo11}. First we recall some results on the index of a pair of projections from \cite{ASS94}. Then we use  Proposition \ref{P1}  to extend some results obtained in  \cite{Pu09}  for semi-bounded operators, to  non-semi-bounded operators (almost without changing the proofs). 

Let $\mathcal H$ be a Hilbert space and let   $P,Q$ be a Fredholm pair, i.e., they are orthogonal projections on $\mathcal H$ satisfying
$$
\sigma_{ess}(P-Q)\cap\{-1,1\}=\emptyset.
$$
Define the index 
$$
{\rm Index}(P,Q)={\rm dim\,Ker}\Big(P-Q-I\Big)-{\rm dim\, Ker}\Big(P-Q+I\Big),
$$
(or equivalently)
\bel{3}
{\rm Index}(P,Q) = {\rm dim}({\rm  Ran}\,P \cap {\rm Ker}\, Q)-{\rm dim}({\rm Ran}\,Q\cap {\rm  Ker}\,P).
\ee

It can be shown 
 that if $P-Q$ is a trace class operator then
\bel{1}
{\rm Index}(P,Q)={\rm Tr}(P-Q).
\ee
Moreover, if both $(P, Q)$ and $(Q, R)$ are Fredholm pairs and at least one of the differences $P-Q$ or $Q-R$ is compact, then $(P, R)$ is also Fredholm and
\bel{2}
{\rm Index}(P, R) = {\rm Index}(P, Q) +{\rm Index}(Q, R) .
\ee

Let $A$ and $B$ be two self-adjoint operators in $\mathcal H$ and 
suppose that for  $\lambda\in\R$ the  pair   $(E_A(-\infty,\lambda)$, $E_B(-\infty,\lambda))$ is a Fredholm pair. For those $\lambda$ define  
\begin{align}
\Xi(\lambda;B,A)&:={\rm Index}\Big(E_A(-\infty,\lambda),E_B(-\infty,\lambda)\Big).
\end{align}

Let us mention that this quantity is defined for $\lambda\notin\sigma_{\rm ess}(A)$ as soon as $B=A+M$, with ${\rm Dom} (B)= {\rm Dom} (A)$ and if $M=M^*$ is $A-$compact, i.e. $ {\rm Dom} (A) \subset {\rm Dom} (M)$ and $M(A-i)^{-1}$ is compact. Indeed thanks to \cite[Corollary 3.5]{Le05}, we have:

\begin{proposition}[\cite{Le05}]\label{P1}
Let $A$ and $B$ be two self-adjoint operators in $\mathcal H$ with same domains such that $M=B-A$ is $A-$compact. Then for $\lambda \in \R\setminus\sigma_{\rm ess}(A)$, the difference $E_{B}(-\infty,\lambda)-E_{A}(-\infty,\lambda)$ is compact and so $\Xi(\lambda;B,A)$ is well defined.
\end{proposition}

For an interval $\mathcal J\subset\R\setminus\sigma_{\rm ess}(A)$ define the eigenvalue counting function 
$$
\mathcal N(\mathcal J;A):= {\rm Tr}\,E_A({\mathcal J})
$$

\begin{lemma}\label{Le42} Assume that $B=A+M$, with $M$ self-adjoint and $A-$compact. Then if $[\lambda_1,\lambda_2]\subset\R\setminus\sigma_{\rm ess}(A)$ we have that
$$
\Xi(\lambda_1;B,A)-\Xi(\lambda_2;B,A)=\mathcal N([\lambda_1,\lambda_2);B)-\mathcal N([\lambda_1,\lambda_2);A).
$$
\end{lemma}
\begin{proof} Since, for $\bullet = A$ or $B$, $E_\bullet(-\infty,\lambda_1)-E_\bullet(-\infty,\lambda_2)=E_\bullet[\lambda_1,\lambda_2)$ is of finite rank, by using \eqref{1} we see that
$${\rm Index}\Big((E_\bullet(-\infty,\lambda_1),E_\bullet(-\infty,\lambda_2)\Big)
={\rm Tr}\Big(E_\bullet[\lambda_1,\lambda_2)\Big)
=\mathcal N\Big([\lambda_1,\lambda_2);\bullet\Big).$$
Moreover, thanks to Proposition \ref{P1}, we can apply \eqref{2} to obtain:
$$ \Xi(\lambda_1;B,A) = \mathcal N\Big([\lambda_1,\lambda_2); B \Big) + \Xi(\lambda_2;B,A) - \mathcal N\Big([\lambda_1,\lambda_2); A \Big),$$
and conclude the proof.
\end{proof}

\begin{lemma}\label{le1} Under the assumption of Lemma \ref{Le42}, for $M_\pm=1/2(|M|\pm M)$ and  $\lambda \in \R\setminus\sigma_{\rm ess}(A)$, we have
$$
-{\rm rank}\,M_-\leq \Xi(\lambda;B,A) \leq {\rm rank}\,M_+
$$
\end{lemma}
\begin{proof}
Let us suppose that ${\rm dim}({\rm  Ran}\,E_A(-\infty,\lambda) \cap {\rm Ker}\, E_B(-\infty,\lambda)) > {\rm dim}\,{\rm rank}\,M_+$, ie., there exists $\psi\neq0$ such that $\psi\in ({\rm  Ran}\,E_A(-\infty,\lambda) \cap {\rm Ran}\, E_B[\lambda,\infty))$ such that $\|\psi\|=1$ and $M_+\psi=0$.
The first assumption is equivalent to say that:
\bel{4}
\langle A \psi,\psi\rangle \in (-\infty,\lambda);\quad \langle B \psi,\psi\rangle \in \R\setminus(-\infty,\lambda)
\ee
on the other side $M_+\psi=0$ implies $$\langle B \psi,\psi\rangle=\langle (A+M )\psi,\psi\rangle \leq \langle A \psi,\psi\rangle,$$
 so we get a contradiction with \eqref{4}.
\end{proof}

\begin{lemma}\label{Le44}[monotonicity]
 Assume that $M_1$ and $M_2$ are self-adjoint and $A-$compact operators. Then  for  $\lambda \in \R\setminus\sigma_{\rm ess}(A)$, we have: $$M_1\geq M_2 \quad \Longrightarrow \quad \Xi(\lambda; A,A+M_1)\geq\Xi(\lambda; A,A+M_2).$$
\end{lemma}
\begin{proof}
Since by proposition \ref{P1} we have that $E_{A+M_1}(-\infty,\lambda)-E_{A}(-\infty,\lambda)$ is compact we can use  \eqref{2} to obtain
\begin{align*}
&{\rm Index}(E_{A+M_1}(-\infty,\lambda),E_{A+M_2}(-\infty,\lambda))&\\
=&{\rm Index}(E_{A+M_1}(-\infty,\lambda),E_{A}(-\infty,\lambda))
+{\rm Index}(E_{A}(-\infty,\lambda),E_{A+M_2}(-\infty,\lambda)).
\end{align*}
Moreover, ${\rm Index}(E_{A}(-\infty,\lambda),E_{A+M_2}(-\infty,\lambda))=-{\rm Index}(E_{A+M_2}(-\infty,\lambda),E_{A}(-\infty,\lambda))$,  and by lemma \ref{le1}  ${\rm Index}(E_{A+M_1}(-\infty,\lambda),E_{A+M_2}(-\infty,\lambda))\geq0$. 
\end{proof}

\subsection{Proof of Proposition \ref{Diag}}\label{ss42}

The following  proof is inspired by \cite{PuRo11}. Note however that {the exact analogue}  of \cite{PuRo11} would give ${\cV}_\epsilon^\pm = {\cV} \pm \epsilon |{\cV}|$, whose expression {for a matrix operator}  is more difficult to handle than the one given in this proposition.

\begin{proof}
	Fix $q\in \Z$. We will make the proof for the upper bound of  $\mathcal N_+^q$.
	First 
	if $A=D_0$ and $M=\cV$, then $A$ and $B=A+M$ satisfy the conditions of Proposition \ref{P1}, and 
	the index $\Xi(\lambda;D_0+\cV,D_0)$ is well defined for
	any $\lambda \in \R\setminus \{\mu_r\}_{r\in\Z}$. Moreover, from Lemma \ref{Le42} 
	$$
	\mathcal N_+^q(\lambda)=\Xi(\mu_{q}+\lambda;D_0+\cV,D_0) -\Xi\left(\alpha;D_0+\cV,D_0\right).
	$$
	Next, set $\mathcal{P}_q^\perp:=Id-\mathcal{P}_q,$ and 
$$
{\cV}_d:= \begin{pmatrix}
V_1&0\\
0&V_2
\end{pmatrix};\qquad {\cV}_a:= \begin{pmatrix}
0&W^*\\
W&0
\end{pmatrix}.
$$	
	
	{
	Thus $ {\cV} = {\cV}_q + {\cV}_a$ 
and in the sense of the quadratic forms in $\C^2$ for a.e. $x \in \R^2$, 
\bel{Vda}
{\cV}(x) \leq |{\cV}_d|(x) + |{\cV}_a|(x), \qquad \text{where}\quad
|{\cV}_d|= \begin{pmatrix}
|{V}_1|&0\\
0&|V_2|
\end{pmatrix}; \quad |{\cV}_a|= |W| I_2 .
\ee
}
Then 
one can see that 
	$$
	{\cV}\leq \mathcal{P}_q ({\cV}+\epsilon (|{\cV} _d|+|{\cV}_a|)\mathcal{P}_q
	+\mathcal{P}_q^\perp ({\cV}+\epsilon^{-1} (|{\cV}_d|+|{\cV}_a|))\mathcal{P}_q^\perp,
	$$
and it is not difficult to see that ${\cV}+\epsilon (|{\cV}_d|+|{\cV}_a|)={\cV}_\epsilon^+$. 	  \Bk
	Thus, by Lemma \ref{Le44} 
	$$
	\Xi(\mu_{q}+\lambda;D_0+{\cV},D_0)\leq\Xi(\mu_{q}+\lambda;D_0+\mathcal{P}_q {\cV}_\epsilon^+\mathcal{P}_q
	+\mathcal{P}_q^\perp ({\cV}+\epsilon^{-1} (|{\cV}_d|+|{\cV}_a|))\mathcal{P}_q^\perp,D_0) 
	$$
	Using again Lemma \ref{Le42} 
	\begin{align*}	
		&\Xi(\mu_{q}+\lambda;D_0+\mathcal{P}_q {\cV}_\epsilon^+\mathcal{P}_q
		+\mathcal{P}_q^\perp ({\cV}+\epsilon^{-1} (|{\cV}_d|+|{\cV}_a|))\mathcal{P}_q^\perp,D_0) \\
		=& \mathcal N((\mu_{q}+\lambda), \alpha) ;D_0+\mathcal{P}_q {\cV}_\epsilon^+\mathcal{P}_q
		+\mathcal{P}_q^\perp ({\cV}+\epsilon^{-1} (|{\cV}_d|+|{\cV}_a|))\mathcal{P}_q^\perp)\\
		-&\Xi(\alpha;D_0+\mathcal{P}_q {\cV}_\epsilon^+\mathcal{P}_q
		+\mathcal{P}_q^\perp (V+\epsilon^{-1} (|{\cV}_d|+|{\cV}_a|))\mathcal{P}_q^\perp, D_0 ) .
	\end{align*}
	Further 
	\begin{align*}
		&\mathcal N((\mu_{q}+\lambda), \alpha) ;D_0+\mathcal{P}_q {\cV}_\epsilon^+\mathcal{P}_q
		+\mathcal{P}_q^\perp ({\cV}+\epsilon^{-1} (|{\cV}_d|+|{\cV}_a|))\mathcal{P}_q^\perp)\\
		=&\mathcal N((\mu_{q}+\lambda, \alpha );\mathcal{P}_q(\mu_{q}+{\cV}_\epsilon^+)\mathcal{P}_q)\\
		+&\mathcal N((\mu_{q}+\lambda, \alpha );\mathcal{P}_q^\perp ({\cV}+\epsilon^{-1} (|{\cV}_d|+|{\cV}_a|))\mathcal{P}_q^\perp).
	\end{align*}
	Finally, by simply noting that $\mathcal{P}_q^\perp D_0 \mathcal{P}_q^\perp = \sum_{p \neq q} \mu_p \mathcal{P}_p $ and each $\mathcal{P}_q^\perp ({\cV}+\epsilon^{-1} (|{\cV}_d|+|{\cV}_a|))\mathcal{P}_q^\perp	$ is compact, we deduce that $(\mu_{q}+\lambda, \alpha )$ contains only a finite number of eigenvalues of $\mathcal{P}_q^\perp(D_0+{\cV}+\epsilon^{-1}(|{\cV}_d|+|{\cV}_a|))\mathcal{P}_q^\perp$ and then 
	$$\mathcal N \left((\mu_{q}+\lambda, \alpha );\mathcal{P}_q^\perp ({\cV}+\epsilon^{-1} (|{\cV}_d|+|{\cV}_a|)\mathcal{P}_q^\perp\right)=O(1),
	$$
	as $\lambda>0$ goes to $0$.
	Putting all this inequalities together we finish the proof.
\end{proof} 

\subsection{Encirclement property}

In this section, in order to justify the relevance of the encirclement property (see Definition \ref{def}), we prove a general encirclement property of an open  connected set  provided there is some space inside its outerboundary. Obviously this condition is not necessary  as is easy  to construct examples of compact sets encircled by an open set not satisfying this separation condition. 

For a open bounded set $\Omega$, we have the decomposition 
\bel{6juin24}
\C\setminus \overline\Omega=\mathcal C_{B}\Omega\cup\mathcal C_{U}\Omega,
\ee
where $\mathcal C_{B}\Omega$ is the union of all the bounded connected components of $\C\setminus \overline\Omega$ and  $\mathcal C_{U}\Omega$ is the unbounded connected component. Notice that   the set $\mathcal S\Omega:=\C\setminus C_{U}\Omega$ is simply connected and $\Omega\subset 	\mathcal S\Omega$.

\begin{proposition}\label{leborder}
	Let $\Omega$ be a bounded connected open set in $\C$.  Suppose that ${\rm dist}(\mathcal C_{B}\Omega,\mathcal C_{U}\Omega)$ is positive. Then any compact set $K$ contained in $	\mathcal S\Omega$ is encircled by $\Omega$.
\end{proposition}
\begin{proof}
By the Riemann mapping theorem there exists a conformal map $\phi:\mathcal S\Omega\to\mathbb D$ such that for some $k_0 \in K$, $\phi(k_0)=0$, where $\mathbb D$ is the unit disk in $\C$. Under the assumption ${\rm dist}(\mathcal C_{B}\Omega,\mathcal C_{U}\Omega)>0$, there exists $r \in (0,1)$ such that 
	$$
	C_r:=\{ z\in \mathbb D: |z|>r\}\subset \phi(\Omega).
	$$
	Indeed, if this is not true there exist a sequence $\{y_n\}\subset \mathbb D$ such that $|y_n|\to1$ and $y_n\notin \phi( \Omega)$ for $n\in\{0,1,2,...\}$. Being the set $\{x_n:=\phi^{-1}(y_n)\}$ bounded, there exists $ (x_m)_m$, a sequence  in $\mathcal S\Omega$ 	such that $x_m\to x\in\overline{\mathcal{S}\Omega}$. Since $|\phi(x_m)|\to 1$ we see that $x\in\partial (\mathcal S\Omega)$. Next, since 	${\rm dist}(C_{B}\Omega,\mathcal C_{U}\Omega)>0$ one can see that $\C\setminus\Omega=\overline{ {\mathcal C}_B \Omega} \cup \overline {{\mathcal C}_U\Omega}$.
	
	By the definition of $\mathcal S\Omega$, if $x_m\in \overline {{\mathcal C}_U\Omega}={\mathcal C}_U\Omega\cup \partial {\mathcal C}_U\Omega$  then $x_m\notin\mathcal S\Omega$ which is not possible. On the other side $x_m\in \overline {{\mathcal C}_B\Omega}$ implies that $x\in \overline {{\mathcal C}_B\Omega}$. But we already know that $x\in\partial (\mathcal S\Omega)=\partial {\mathcal C}_U\Omega$  and $ {\rm dist }(\partial {\mathcal C}_U\Omega,\overline {{\mathcal C}_B\Omega})>0$. Its gives again a contradiction.

	Now, take $\delta\in (r,1)$ such that for any $z\in \phi(K)$, $|z|<\delta$. Set the curve in $\Omega$ given by $\Gamma:=\phi^{-1}(\{z:|z|=\delta \})	$. Since $\phi(k_0)=0$ for a point $k_0 \in K$, we immediately obtain that $K$ is in the interior part of $\Gamma$. \end{proof}

%
%
%
%

{\bf Sources of funding}. 
This work has received the financial support from the French government in the framework of the France 2030 programme IdEx universit\'e de Bordeaux. 
Pablo Miranda was partially supported by Fondecyt grant 1241983. Vincent Bruneau was partially supported by ANR grant ANR-24-CE40-2939-01.

{\bf Acknowledgements} The authors thank Jean-Fran\c cois  Bony and Galina Levitina for useful discussions and ideas.

{\sc Vincent Bruneau}\\ Institut de Math\'ematiques de Bordeaux, UMR  CNRS 5251\\
Universit\'e de Bordeaux\\ 351 cours de la Lib\'eration, 33405 Talence cedex, France\\
E-mail: vbruneau@math.u-bordeaux.fr\\

{\sc Pablo Miranda}\\ Departamento de Matem\'atica y Ciencia de la Computaci\'on\\
Universidad de Santiago de Chile\\ Las Sophoras 173, Estaci\'on Central, Santiago, Chile\\
E-mail: pablo.miranda.r@usach.cl\\

\end{document}